\documentclass[a4paper,leqno,11pt]{article}

\usepackage{amsmath,amssymb,amsthm,
verbatim,
fancyhdr,amsfonts
}

\usepackage{tikz}
\usetikzlibrary{decorations.pathreplacing}

\newcommand{\cor}[1]{\textit{#1}} 
\newcommand{\gr}[1]{\textbf{#1}} 
\newcommand{\m}[1]{\mathrm{#1}} 
\newcommand\R{{\mathbb R}} 
\newcommand\Rp{{\mathbb R}^+} 
\newcommand\implica{\quad\Rightarrow\quad} 
\newcommand\f{\frac} 
\newcommand\tf{\tfrac} 
\newcommand\der[2]{\f{d#1}{d#2}} 
\newcommand\iii{\infty} 
\newcommand\intd[3]{\int_{#1}#2d#3} 
\newcommand\intl[4]{\int_{#1}^{#2}#3d#4} 
\newcommand\mo[1]{|#1|} 
\newcommand\mis{\mathop{\rm meas}} 
\newcommand\ff[1]{\varphi_{#1}} 
\newcommand\pp{\psi} 
\newcommand\gn{\mathcal{C}_{\iii}} 
\newcommand\ru{u^*} 
\newcommand\sech{\mathrm{sech}} 
\newcommand\G{\mathcal G} 
\newcommand\K{\mathcal K} 
\newcommand\V{\m{V}} 
\newcommand\Vk{\m{V}_\K} 
\newcommand\E{\m{E}} 
\newcommand\Ek{\m{E}_\K} 
\newcommand\I{I} 
\newcommand\lung{L} 
\newcommand\nn{N} 
\newcommand\pa{\nu} 
\newcommand\vg{\textsc{v}} 
\newcommand\vv[1]{\textsc{v}_{#1}} 
\newcommand\w{\textsc{w}} 
\newcommand\Rv[1]{\mathcal{R}_{#1}} 
\newcommand\thc{\widetilde{\lung}} 
\newcommand\Ce{\mathcal{C}_p} 
\newcommand\Cp{c_p} 
\newcommand\thes{\lung_{1}} 
\newcommand\thnon{\lung_{2}} 
\newcommand\ldueg{{L^2(\G)}} 
\newcommand\nldueg[1]{\left\|#1\right\|_\ldueg} 
\newcommand\lduegc{{L^2(\K)}} 
\newcommand\nlduegc[1]{\left\|#1\right\|_\lduegc} 
\newcommand\lpg{{L^p(\G)}} 
\newcommand\nlpg[1]{\left\|#1\right\|_\lpg} 
\newcommand\lpgc{{L^p(\K)}} 
\newcommand\nlpgc[1]{\left\|#1\right\|_\lpgc} 
\newcommand\ldues{L^2(\Rv{i})} 
\newcommand\nldues[1]{\left\|#1\right\|_{\ldues}} 
\newcommand\lig{{L^{\iii}(\G)}} 
\newcommand\nlig[1]{\left\|#1\right\|_\lig} 
\newcommand\ldue{{L^2(\R)}} 
\newcommand\nldue[1]{\left\|#1\right\|_\ldue} 
\newcommand\li{{L^{\iii}(\R)}} 
\newcommand\nli[1]{\left\|#1\right\|_\li} 
\newcommand\lldue{{L^2(\R^+)}} 
\newcommand\nlldue[1]{\left\|#1\right\|_\lldue} 
\newcommand\llp{{L^p(\R^+)}} 
\newcommand\nllp[1]{\left\|#1\right\|_\llp} 
\newcommand\hg{{H^1(\G)}} 
\newcommand\hk{{H^1(\K)}} 
\newcommand\hrv[1]{{H^1(\Rv{#1})}} 
\newcommand\hm{{H_{\mu}^1(\G)}} 
\newcommand\nhg[1]{\left\|#1\right\|_\hg} 
\newcommand\hr{{H^1(\R)}} 
\newcommand\hrpos{{H^1(\Rp)}} 
\newcommand\hrm{{H_{m}^1(\R)}} 
\newcommand\hdir{{H_{m,a}^1(\R)}} 
\newcommand\hhdir{{H_{m,a}^1(\R^+)}} 
\newtheorem{theorem}{Theorem}[section]
\newtheorem{proposition}[theorem]{Proposition}

\newtheorem{corollary}[theorem]{Corollary}

\theoremstyle{remark}
\newtheorem{remark}[theorem]{Remark}
\newtheorem*{remark*}{Remark}

\theoremstyle{definition}
\newtheorem{definition}[theorem]{Definition}
\newtheorem{example}[theorem]{Example}

\tikzstyle{nodino}=[circle,draw,fill,inner sep=0pt,minimum size=0.5mm]
\tikzstyle{infinito}=[circle,inner sep=0pt,minimum size=0mm]
\tikzstyle{nodo}=[circle,draw,fill,inner sep=0pt,minimum size=0.5*\widthof{k}]

\date{}

\title{NLS ground states on metric graphs\\with localized nonlinearities}

\author{Lorenzo Tentarelli
\\ \ \\
{\small  Dipartimento di Scienze Matematiche ``G.L. Lagrange'', Politecnico di Torino } \\
{\small Corso Duca degli Abruzzi, 24, 10129 Torino, Italy} \\
{\small \texttt{lorenzo.tentarelli@polito.it}}}

\begin{document}

\maketitle

\begin{abstract}
We investigate the existence of ground states for the focusing subcritical NLS energy on metric graphs with localized nonlinearities. In particular, we find two thresholds on the measure of the region where the nonlinearity is localized that imply, respectively, existence or nonexistence of ground states. In order to obtain these results we adapt to the context of metric graphs some classical techniques from the Calculus of Variations.
\end{abstract}

\smallskip

\noindent{\small Keywords: Minimization, metric graphs, NLS energy, nonlinear Schr$\ddot{\m{o}}$dinger\\Equation, localized nonlinearity.}

\section{Introduction}\label{sec:intro}

In this paper we discuss the existence of a ground state for the NLS energy functional with a \cor{localized nonlinearity}
\begin{equation}
 \label{eq:levelfunc}
 E(u,\G):=\f{1}{2}\nldueg{u'}^2-\f{1}{p}\nlpgc{u}^p,
\end{equation}
where $p\in(2,6)$ is a real parameter and $\G$ and $\K$ are two connected \cor{metric graphs}, such that $\K$ is a \cor{compact} subgraph of $\G$. Namely, we investigate the minimization problem
\begin{equation}
 \label{eq:minimipbm}
 \min E(u,\G),\quad u\in\hg,\quad\nldueg{u}^2=\mu
\end{equation}
where $\mu>0$ is a given number.

\medskip
Here we present a rather informal description of the problem and of the main results of the paper, whereas a precise setting and formal definitions are given in Sections \ref{sec:setting} and \ref{sec:results}.

It is well known that, when the graph $\G$ is compact, the minimization problem is trivial. The aim of this paper, therefore, is the study of existence and nonexistence of solutions to (\ref{eq:minimipbm}) when $\G$ is made up of a compact graph $\K$ where the nonlinearity is localized and a finite number of \cor{half-lines} $\Rv{i}$ incident at some vertices of $\K$. For the sake of completeness, we recall that the case of \cor{non-localized} nonlinearities (where the \cor{lack of compactness} is stronger) is investigated in \cite{AST-1}.

The main results we are going to present are the following. In Theorem \ref{theorem:infsotto} we adapt to the context of metric graphs a classical \cor{level argument} (see e.g. \cite{BN} and \cite{Prinari, Bellazzini}) in order to overcome the lack of compactness. More in detail, first we show that under general assumptions
\[
 \inf_{\substack{u\in\hg\\[.1cm]\nldueg{u}^2=\mu}}E(u,\G)\leq0
\]
and further that, if
\[
 \inf_{\substack{u\in\hg\\[.1cm]\nldueg{u}^2=\mu}}E(u,\G)<0,\quad\mbox{then the infimum is attained}.
\]
Consequently, we exploit these two properties to prove that existence and nonexistence of ground states strongly depends on the exponent $p$ of the nonlinearity and on the metric properties of the compact graph $\K$. In particular, Theorem \ref{theorem:pdue} states that, if $p\in(2,4)$, then there always exists a minimizer for $E$, whereas Theorem \ref{theorem:pquattro} states that, when $p\in[4,6)$, the existence of minimizers depends on the measure of $\K$; precisely, there is a threshold for $\mis(\K)$ over which the minimum is always attained and another one under which the minimum cannot be attained. In Corollary \ref{corollary:nonesistenza} we slightly improve the nonexistence threshold, by adding some further assumptions on the inner structure of $\G$.

Finally we highlight that all the functions we consider are real-valued. Indeed, as pointed out in \cite{AST-1}, this is not restrictive since $E(|u|,\G)\leq E(u,\G)$ and any ground state is in fact real-valued (and then, without loss of generality, nonnegative), up to multiplication by a constant phase $e^{i\vartheta}$.

\medskip
Among the physical motivations for the investigation of this kind of problem, nowadays the most topical ones come from the study of \cor{quantum graphs}. This subject has gained popularity in recent years (as extensively pointed out in \cite{Kottos, Gnutzmann1}) not only because graphs emulate succesfully complex mesoscopic and optical networks, but also because they manage to reproduce universal properties of quantum chaotic systems. In particular, in \cite{Gnutzmann2} (and the references therein) two main applications are exhibited to motivate the study of NLS energy functionals with localized nonlinearities. Indeed, they can be of interest both in the analysis of the effects of nonlinearity on transmission through a complex network of nonlinear one-dimensional leads (for istance, \cor{optical fibers}) and in the analysis of the properties of \cor{Bose-Einstein condensates} in non-regular traps (see also \cite{AST-1,Dalfovo}). In other words, even though these are idealized models, they actually describe 
systems made up by several long leads (where the nonlinearity is, in principle, negligible) linked to a ``dense'' sub-network (where the nonlinearity is strong) and reproduce accurately the effects that the latter induces on the whole network.

It is also worth recalling that the investigation of the NLS equation with localized nonlinearities in standard domains has been the object of several papers. In the one-dimensional case we first mention \cite{Malomed}, that introduces a pioneering model of nonlinearities concentrated at a finite number of points. Then, \cite{Adami3} contains a rigorous definition of the problem with potentials modeled by \cor{Dirac delta functions} with amplitude depending nonlinearly on the wave function.  Furthemore, \cite{Cacciapuoti} shows that such nonlinearities can be seen as \cor{point-like} limits of spatially concentrated nonlinearities, whereas \cite{Buslaev} deals with the associated issue of the asymptotic stability. Finally, for a rigorous setting of the case of a \cor{point-interaction} with strength depending nonlinearly on the wave function in dimension three, we refer the reader to \cite{Adami1}, while in \cite{Jonalasinio,Nier} one can find several models of concentrated nonlinearities in dimension three 
or higher.

In this framework, our work can be seen as a slight extension of the one-dimensional case (although we only focus on ground states of the NLS energy) in that a nonlinear potential is spread over a compact graph in place of a Dirac delta based at single point.

\medskip
The paper is organized as follows. In Sections \ref{sec:setting} and \ref{sec:results} we present, respectively, the precise setting of the problem and the statemets of the main results. Section \ref{sec:aux} contains a discussion of some preliminary and auxiliary facts and techniques that may have interest in themselves. Finally, in Section \ref{sec:proofs} we prove the results stated in Section \ref{sec:results}.

\section{Setting and notation}\label{sec:setting}

Although it is a central notion in this paper, we do not present a deep and detailed discussion on \cor{metric graphs}. For a modern account on the subject we refer the reader to \cite{AST-1,Friedlander,Kuchment} and references therein.

We start with a brief recall on metric graphs and, in particular, on the class of metric graphs that we focus on in this paper. First of all, the graphs that we consider are \cor{multigraphs}, namely, graphs with possibly multiple edges and self-loops. Moreover we only deal with \cor{connected metric graphs}, that is, connected graphs with the ``representation'' introduced in \cite{AST-1}, which associates each edge $e$ of length $l_e$ with an interval $\I_e:=[0,l_e]$ parametrized by a variable $x_e$ ($x_e=0$ representing either the starting or the endpoint of $e$, depending on the orientation choice). In addition, if $l_e=+\iii$, then $e$ is called \cor{half-line} and $\I_e:=[0,+\iii)$ (now $x_e=0$ representing the starting point of the half-line).

Throughout this paper we denote by $\K=(\Vk,\Ek)$ a (non trivial) \cor{compact} connected metric graph, i.e. a connected metric graph without half-lines (or, equivalently, without \cor{vertices at infinity}). We stress the fact that there are no constraints on the degree of the vertices. This means that there are possibly vertices of degree one or two (see, for instance, Figure \ref{figuno}).

Furthermore, we consider $\nn$ distinct half-lines $\Rv{1},\dots,\Rv{\nn}$ starting from vertices (possibly not distinct) in $\K$ and ending, respectively, into $\nn$ (distinct) vertices at infinity $\vv{1},\dots,\vv{\nn}$. Then we define the graph
\begin{equation}
 \label{eq:graphdef}
 \begin{array}{c}
  \G:=(\V,\E),\\[.3cm]
  \V:=\Vk\cup\{\vv{1},\dots\vv{\nn}\},\quad\E:=\Ek\cup\{\Rv{1},\dots,\Rv{\nn}\}
 \end{array}
\end{equation}
(a typical case is depicted in Figure \ref{figdue}).

\begin{figure}[ht]
\begin{center}
\begin{tikzpicture}[xscale= 0.7,yscale=0.7]
\node at (0,0) [nodo] (1) {};
\node at (2,0) [nodo] (2) {};
\node at (4,1) [nodo] (3) {};
\node at (2,3) [nodo] (4) {};
\node at (1,2) [nodo] (5) {};
\node at (-1,2) [nodo] (6) {};
\node at (-2,0) [nodo] (7) {};
\node at (-1,3) [nodo] (8) {};
\node at (-2,2) [nodo] (9) {};
\node at (-4,0) [nodo] (10) {};
\node at (-2,-2) [nodo] (11) {};
\node at (0,-2) [nodo] (12) {};
\node at (3,-2) [nodo] (13) {};
\node at (-3,1) [nodo] (14) {};

\draw (4.5,1) circle (0.5cm);

\draw [-] (1) -- (2);
\draw [-] (2) -- (3);
\draw [-] (3) -- (4);
\draw [-] (4) -- (5);
\draw [-] (5) -- (2);
\draw [-] (1) -- (13);
\draw [-] (1) -- (12);
\draw [-] (5) -- (6);
\draw [-] (1) -- (6);
\draw [-] (8) -- (6);
\draw [-] (11) -- (12);
\draw [-] (7) -- (12);
\draw [-] (1) -- (7);
\draw [-] (6) -- (7);
\draw [-] (9) -- (7);
\draw [-] (7) -- (11);
\draw [-] (11) -- (10);
\draw [-] (7) -- (10);
\draw [-] (9) -- (10);
\draw [-] (9) -- (6);
\draw [-] (1) to [out=100,in=190] (5);
\draw [-] (1) to [out=10,in=270] (5);
\end{tikzpicture}
\end{center}
\caption{\footnotesize{{a \cor{compact metric graph} $\K$ with 14 vertices and 24 bounded edges. In particular there are: 1 \cor{self-loop}, 1 \cor{multiple} edge, 2 vertices of \cor{degree one} and 2 vertices of \cor{degree two}.}
 }}
\label{figuno}
\end{figure}
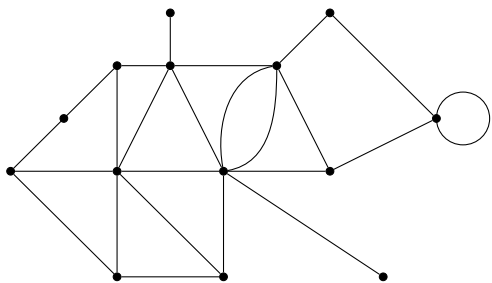

\begin{figure}[ht]
\begin{center}
\begin{tikzpicture}[xscale= 0.7,yscale=0.7]
\node at (0,0) [nodo] (1) {};
\node at (2,0) [nodo] (2) {};
\node at (4,1) [nodo] (3) {};
\node at (2,3) [nodo] (4) {};
\node at (1,2) [nodo] (5) {};
\node at (-1,2) [nodo] (6) {};
\node at (-2,0) [nodo] (7) {};
\node at (-1,3) [nodo] (8) {};
\node at (-2,2) [nodo] (9) {};
\node at (-4,0) [nodo] (10) {};
\node at (-2,-2) [nodo] (11) {};
\node at (0,-2) [nodo] (12) {};
\node at (3,-2) [nodo] (13) {};
\node at (-3,1) [nodo] (14) {};
\node at (8,-2) [nodo] (15) {};
\node at (8,3) [nodo] (16) {};
\node at (-9,-2) [nodo] (17) {};
\node at (-8.5,4) [nodo] (18) {};
\node at (-9,1) [nodo] (19) {};
\node at (8,-1.6) {$\vv{1}$};
\node at (8,3.4) {$\vv{3}$};
\node at (-8.4,3.5) {$\vv{2}$};
\node at (-9,-1.6) {$\vv{4}$};
\node at (-9,0.6) {$\vv{5}$};
\node at (5.5,-1.6) {$\Rv{1}$};
\node at (5,3.4) {$\Rv{3}$};
\node at (-5.6,2.9) {$\Rv{2}$};
\node at (-5.5,-1.6) {$\Rv{4}$};
\node at (-6,0.2) {$\Rv{5}$};
\node at (-2,3) {$\K$};

\draw (4.5,1) circle (0.5cm);

\draw [-] (1) -- (2);
\draw [-] (2) -- (3);
\draw [-] (3) -- (4);
\draw [-] (4) -- (5);
\draw [-] (5) -- (2);
\draw [-] (1) -- (13);
\draw [-] (1) -- (12);
\draw [-] (5) -- (6);
\draw [-] (1) -- (6);
\draw [-] (8) -- (6);
\draw [-] (11) -- (12);
\draw [-] (7) -- (12);
\draw [-] (1) -- (7);
\draw [-] (6) -- (7);
\draw [-] (9) -- (7);
\draw [-] (7) -- (11);
\draw [-] (11) -- (10);
\draw [-] (7) -- (10);
\draw [-] (9) -- (10);
\draw [-] (9) -- (6);
\draw [-] (1) to [out=100,in=190] (5);
\draw [-] (1) to [out=10,in=270] (5);
\draw [-] (13) -- (5.5,-2);
\draw [dashed] (15) -- (5.5,-2);
\draw [-] (4) -- (5,3);
\draw [dashed] (16) -- (5,3);
\draw [-] (14) -- (-5.75,2.5);
\draw [dashed] (18) -- (-5.75,2.5);
\draw [-] (11) -- (-5.5,-2);
\draw [dashed] (17) -- (-5.5,-2);
\draw [-] (11) -- (-5.5,-0.5);
\draw [dashed] (19) -- (-5.5,-0.5);
\end{tikzpicture}
\end{center}
\caption{\footnotesize{{a noncompact metric graph $\G$ obtained attaching 5 \cor{half-lines} (and 5 \cor{vertices at infinity}) to the graph $\K$ of Figure \ref{figuno}. Note that $\Rv{1}$ is attached to a vertex of degree one, $\Rv{4}\text{ and }\Rv{5}$ to a vertex of degree three and $\Rv{2}\text{ and }\Rv{3}$ to distinct vertices of degree two.}
 }}
\label{figdue}
\end{figure}
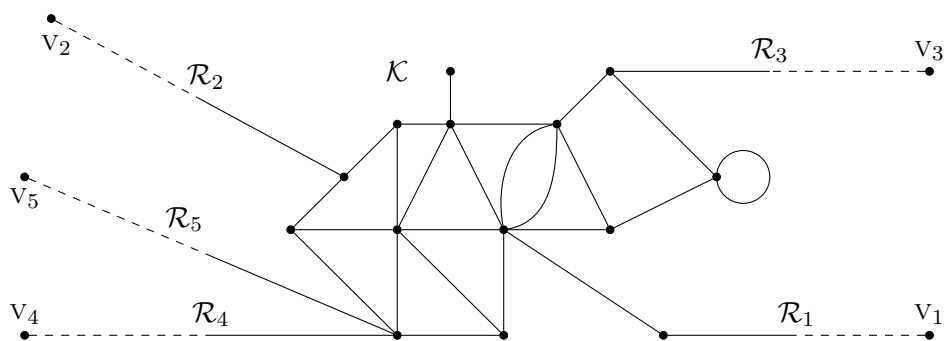

Clearly the measure of the compact graph $\K$, defined by
\begin{equation}
 \label{eq:lunghezzagrafo}
 \mis(\K):=\sum_{e\in\Ek}\mis(\I_e),
\end{equation}
is finite, while that of $\G$ is equal to $\iii$ (and so $\G$ is not compact). Consequently, a function $u:\G\to\R$ can be regarded as a family of functions $(u_e)_{e\in \E}$, where $u_e:\I_e\to\R$ is the restriction of $u$ to the edge $\I_e$, and $L^p$ spaces can be defined over $\G$ in the natural way, with norm
\[
\nlpg{u}^p=\sum_{e\in \E} \| u_e\|_{L^p(\I_e)}^p.
\]
The Sobolev space $\hg$ can be defined as the set of those functions $u:\G\to\R$ such that $u=(u_e)_{e\in \E}$ is continuous on $\G$ and $u_e\in H^1(\I_e)$
for every edge $e\in E$, with the natural norm
\[
\nhg{u}^2:=\nldueg{u'}^2+\nldueg{u}^2.
\]

Fix, now, the graphs $\K$ and $\G$ as above and let $\mu$ and $p$ be such that
\begin{equation}
 \label{eq:paramassum}
 \mu>0\quad\mbox{and}\quad2<p<6.
\end{equation}
Then the functional in (\ref{eq:levelfunc}) is well defined over $\hg$ and takes the form
\[
 E(u,\G)=\f{1}{2}\sum_{e\in\E}\intd{\I_e}{\mo{u_e'(x_e)}^2}{x_e}-\f{1}{p}\sum_{e\in\Ek}\intd{\I_e}{\mo{u_e(x_e)}^p}{x_e}.
\]
Moreover, defining
\begin{equation}
 \label{eq:sobolevmass}
 \hm:=\{u\in\hg:\nldueg{u}^2=\mu\},
\end{equation}
the minimization problem (\ref{eq:minimipbm}) can be written in a compact form as
\begin{equation}
 \label{eq:constrainedminimi}
 \min_{u\in\hm}E(u,\G).
\end{equation}

\begin{remark}
 \label{remark:scaling}
 Note that if $u\in\hm$ and $\lambda$ is positive, then the function
 \[
  w(x):=\lambda^{\frac{2}{6-p}}u\left(\lambda^{\frac{p-2}{6-p}}x\right)
 \]
 belongs to $H_{\lambda\mu}^1(\G')$, where $\G'$ is a graph obtained from $\G$ by means of a \cor{homothety} of factor $\lambda^{\frac{2-p}{6-p}}$. Moreover a change of variable shows that
 \[
  E(w,\G')=\lambda^{\frac{2+p}{6-p}}E(u,\G),
 \]
 and hence the problem with mass constraint $\lambda\mu$ over the graph $\G'$ is equivalent to problem (\ref{eq:constrainedminimi}).
\end{remark}

\section{Main results}\label{sec:results}

In this section we present the main results on \cor{existence} and \cor{nonexistence} of a solution for problem (\ref{eq:constrainedminimi}).

We start with a general result.

\begin{theorem}
 \label{theorem:infsotto}
 Let $\G$ be a graph as in (\ref{eq:graphdef}) and let $\mu,p$ satisfy (\ref{eq:paramassum}). Then
 \begin{equation}
  \label{eq:infzero}
  \inf_{u\in\hm}E(u,\G)\leq0.
 \end{equation}
 Moreover, if
 \begin{equation}
  \label{eq:infsotto}
  \inf_{u\in\hm}E(u,\G)<0,
 \end{equation}
 then the infimum is attained and, up to a change of sign, every minimizer is strictly positive.
\end{theorem}

It is easy to construct an example of graph $\G$ for which the infimum is achieved.

\begin{example}
 \label{example:similsoliton}
 Let $p,\mu$ satisfy (\ref{eq:paramassum}) and define the function
 \[
  \varphi_{\mu}(x):=\mu^{\alpha}\varphi_1(\mu^{\beta}x),\quad\alpha=\tf{2}{6-p},\quad\beta=\tf{p-2}{6-p},
 \]
 where
 \[
  \varphi_1(x):=\Gamma_p\sech(\gamma_px)^{\alpha/\beta},\quad\mbox{with}\quad\Gamma_p,\gamma_p>0.
 \]
 It is well known that $\varphi_{\mu}$ (usually called a \cor{solition of mass} $\mu$) is a minimizer for (\ref{eq:constrainedminimi}) when $\G=\R$ and the nonlinearity is \gr{not} localized and, in particular,
 \[
  \f{1}{2}\intd{\R}{\mo{\varphi_{\mu}'(x)}^2}{x}-\f{1}{p}\intd{\R}{\mo{\varphi_{\mu}(x)}^p}{x}=\mathcal{E}<0
 \]
 (for more details, see \cite{Adami2,AST-1}). Now let $\K$ be a segment of length $\lung$ and assume, in addition, that $\nn=2$ and that the two half-lines are incident at the endpoints of $\K$. Hence the graph $\G$ can be seen as a straight line and the energy functional reads
 \[
  E(u,\G):=\f{1}{2}\intd{\R}{\mo{u'(x)}^2}{x}-\f{1}{p}\int_{-\lung/2}^{\lung/2}\mo{u(x)}^pdx.
 \]
 Hence, if $\thc$ is a positive constant such that
 \[
  \intd{\mo{x}>\thc/2}{\mo{\varphi_{\mu}(x)}^p}{x}<\mo{\mathcal{E}},
 \]
 then we see that
 \[
  E(\varphi_{\mu},\G)<0\qquad\mbox{whenever}\qquad\lung>\thc
 \]
 and, from Theorem \ref{theorem:infsotto}, $E$ has a minimizer in $\hm$. \endproof
\end{example}

In fact, in more general cases, existence or nonexistence of minimizers for $E$ strongly depends on the exponent of the nonlinearity and on the measure of the compact graph $\K$. This is described in the following two theorems.

\begin{theorem}
 \label{theorem:pdue}
 Let $\G$ be as in (\ref{eq:graphdef}) and $\mu>0$. Then for every $p\in(2,4)$ the minimum in (\ref{eq:constrainedminimi}) is attained.
\end{theorem}

\begin{theorem}
 \label{theorem:pquattro}
 Let $\G$ be as in (\ref{eq:graphdef}) and $\mu>0$. If $p\in[4,6)$, then there exist two positive constants $\thes,\thnon$ such that
 \begin{equation}
  \label{eq:esistenza}
  \mis(\K)>\thes\implica\mbox{the minimum in (\ref{eq:constrainedminimi}) is attained},
 \end{equation}
 \begin{equation}
  \label{eq:nonesistenza}
  \mis(\K)<\thnon\implica\mbox{the minimum in (\ref{eq:constrainedminimi}) is not attained}.
 \end{equation}
\end{theorem}

In some cases we can improve the threshold that guarantees nonexistence. To this aim, we give the following definition.

\begin{definition}
 \label{definition:partizione}
 Let $\nn\geq2$ and $\G$ be a graph satisfying (\ref{eq:graphdef}). In addition, let $(\G_i)_{i=1}^{\pa}$, with $2\leq\pa\leq\nn$, be a family of non trivial subgraphs of $\G$ (each consisting of edges and vertices of $\G$). We say that $(\G_i)_{i=1}^{\pa}$ is a \cor{partition of} $\G$ if:
 \begin{enumerate}
  \item $\displaystyle\bigcup_{i=1}^{\pa}\G_i=\G$;
  \item $\mis(\G_i\cap\G_l)=0$ whenever $i\neq l$;
  \item for every $i\in\{1,\dots,\pa\}$, there exists $j\in\{1,\dots,\nn\}$ such that $\Rv{j}\subset\G_i$.
 \end{enumerate}
\end{definition}

\noindent In other words, a partition of a graph is the union of ``essentially'' pairwise disjoint subsgraphs, each containing at least a half-line. We also stress the fact that every graph admits, in principle, several distinct partitions. Two examples of partitions are given in Figure \ref{figtre} in the case of a \cor{double bridge} graph.

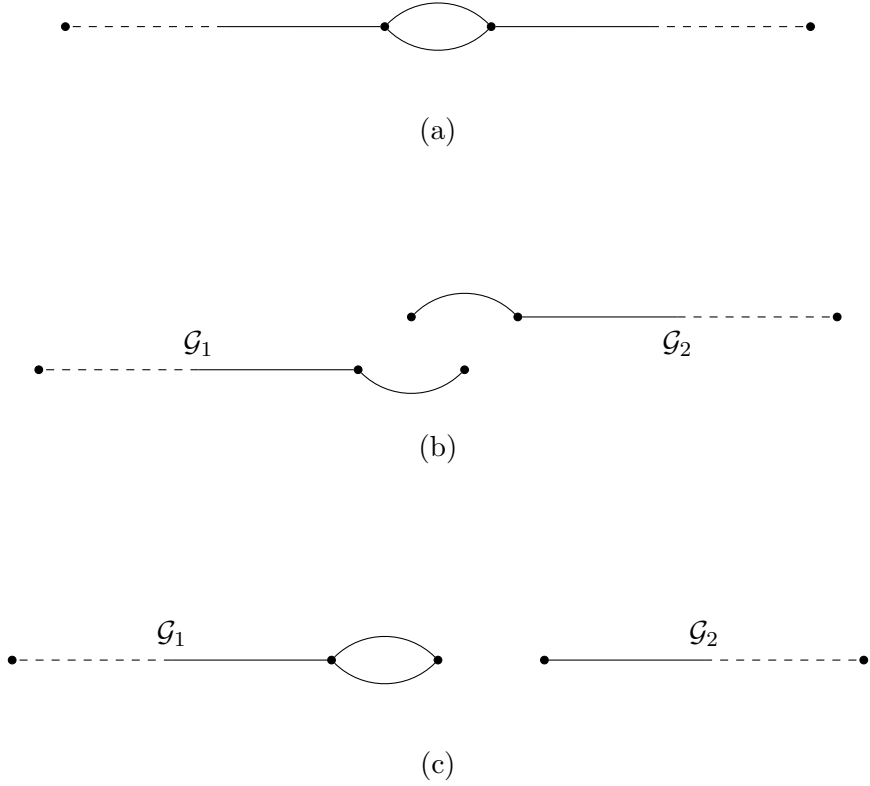
\begin{figure}[ht]
\begin{center}
\begin{tikzpicture}[xscale= 0.7,yscale=0.7]

\node at (-1,6) [nodo] (7) {};
\node at (1,6) [nodo] (8) {};
\node at (-7,6) [nodo] (9) {};
\node at (7,6) [nodo] (10) {};
\node at (0,4) {(a)};

\draw [-] (7) to [out=45,in=135] (8);
\draw [-] (7) to [out=-45,in=-135] (8);
\draw [-] (8) -- (4,6);
\draw [dashed] (4,6) -- (10);
\draw [-] (7) -- (-4,6);
\draw [dashed] (-4,6) -- (9);


\node at (-0.5,0.5) [nodo] (1) {};
\node at (1.5,0.5) [nodo] (2) {};
\node at (7.5,0.5) [nodo] (3) {};
\node at (0.5,-0.5) [nodo] (4) {};
\node at (-1.5,-0.5) [nodo] (5) {};
\node at (-7.5,-0.5) [nodo] (6) {};
\node at (4.5,0) {$\G_2$};
\node at (-4.5,0) {$\G_1$};
\node at (0,-2) {(b)};

\draw [-] (1) to [out=45,in=135] (2);
\draw [-] (2) -- (4.5,0.5);
\draw [dashed] (4.5,0.5) -- (3);
\draw [-] (5) to [out=-45,in=-135] (4);
\draw [-] (5) -- (-4.5,-0.5);
\draw [dashed] (-4.5,-0.5) -- (6);


\node at (-2,-6) [nodo] (11) {};
\node at (0,-6) [nodo] (12) {};
\node at (-8,-6) [nodo] (13) {};
\node at (2,-6) [nodo] (14) {};
\node at (8,-6) [nodo] (15) {};
\node at (5,-5.5) {$\G_2$};
\node at (-5,-5.5) {$\G_1$};
\node at (0,-8) {(c)};

\draw [-] (11) to [out=45,in=135] (12);
\draw [-] (11) to [out=-45,in=-135] (12);
\draw [-] (11) -- (-5,-6);
\draw [dashed] (-5,-6) -- (13);
\draw [-] (14) -- (5,-6);
\draw [dashed] (5,-6) -- (15);

\end{tikzpicture}
\end{center}
\caption{\footnotesize{{(a) a \cor{double bridge} graph; (b) a partition where each $\G_i$ contains a half-line and a bounded edge; (c) a partition where $\G_1$ contains a half-line and all the bounded edges and $\G_2$ only a half-line.}
 }}
\label{figtre}
\end{figure}

Now we can extend (\ref{eq:nonesistenza}) as follows.

\begin{corollary}
 \label{corollary:nonesistenza}
 Let $\G$, $\mu$ and $p$ satisfy the assumptions of Theorem \ref{theorem:pquattro}. In addition, let $(\G_i)_{i=1}^{\pa}$ be a partition of $\G$ and $\K_i=\G_i\cap\K$, for $i=1,\dots,\pa$. If $\thnon$ is the same of Theorem \ref{theorem:pquattro} and 
 \[
  \mis(\K_i)<\thnon\qquad\forall i=1,\dots,\pa,
 \]
 then the minimum in (\ref{eq:constrainedminimi}) is not attained.
\end{corollary}

\begin{remark}
 \label{remark:improvement}
 In order to understand how Corollary \ref{corollary:nonesistenza} ``improves'' Theorem \ref{theorem:pquattro}, consider again the double bridge graph. If one chooses the partition of Figure \ref{figtre}(b), then
 \[
  \mis(\K_1)=\mis(\K_2)=\f{\mis(\K)}{2}.
 \]
 Hence, although $\mis(\K)>L_2$, one can claim that ground states do not exist (provided $\mis(\K_i)=\mis(\K_2)<L_2$).
\end{remark}

\begin{remark}
 \label{remark:cost}
 The bounds $\thes$ and $\thnon$ deserve some further comment. As we will show in Section \ref{sec:proofs}, we can compute them explicitly:
 \begin{equation}
  \label{eq:const}
  \begin{array}{ll}
   \mu\thes=\f{\nn^2}{2} & \mbox{if}\quad p=4 \\[.3cm]
   \mu^{\f{p-2}{6-p}}\thes=\Cp\nn^{\f{4}{6-p}} & \mbox{if}\quad p\in(4,6)\\[.3cm]
   \mu^{\f{p-2}{6-p}}\thnon=\Ce^{\tf{4-p}{6-p}}\gn^{-p} & \mbox{if}\quad p\in[4,6),
  \end{array}
 \end{equation}
 where $\Cp$ is defined by (\ref{eq:const_ex}),
 while $\Ce$ and $\gn$ come, respectively, from inequalities (\ref{eq:gnp}) and (\ref{eq:gni}). In addition one can check that they are invariant under the transformations
 \[
  \mu\rightarrow\lambda\mu\qquad\G\rightarrow\lambda^{\f{2-p}{6-p}}\G,
 \]
 coherently with Remark \ref{remark:scaling}. This entails that the results we stated in Theorems \ref{theorem:pdue} and \ref{theorem:pquattro} can be equivalently formulated in terms of the mass $\mu$ in place of $\mis(\K)$. Precisely, for $\mis(\K)>0$ fixed, if $p\in(2,4)$, then there always exists a solution to (\ref{eq:constrainedminimi}); whereas, if $p\in[4,6)$, then there exist $\mu_1,\mu_2>0$ such that
 \[
  \mu>\mu_1\implica\mbox{the minimum in (\ref{eq:constrainedminimi}) is attained},
 \]
 \[
  \mu<\mu_2\implica\mbox{the minimum in (\ref{eq:constrainedminimi}) is not attained}.
 \]
\end{remark}

\section{Preliminary and auxiliary results}\label{sec:aux}

In this section we introduce some tools that we will use in the proofs of the main results in Section \ref{sec:proofs}.

The first one is a  version of the Gagliardo-Nirenberg inequalites over a graph $\G$.

\begin{proposition}
 \label{proposition:gn}
 Let $\G$ be as in (\ref{eq:graphdef}). For every $p\in[2,\iii]$ there exist two positive constants $\Ce$ (depending only on $p$) and $\gn$ such that
 \begin{equation}
  \label{eq:gnp}
  \nlpg{u}^p\leq \Ce\nldueg{u}^{\f{p}{2}+1}\nldueg{u'}^{\f{p}{2}-1}\quad\forall u\in\hg, \quad\mbox{if}\quad p<\iii,
 \end{equation}
 \begin{equation}
  \label{eq:gni}
  \nlig{u}\leq\gn\nldueg{u}^{\f{1}{2}}\nldueg{u'}^{\f{1}{2}}\quad\forall u\in\hg.
 \end{equation}
\end{proposition}

\begin{proof}
 Consider a nonnegative function $u\in\hg$. We define its \cor{decreasing rearrangement} as the function $\ru:\Rp\rightarrow\R$ such that
 \[
  \ru(x):=\inf\{t\geq0:\rho(t)\leq x\},
 \]
 where
 \[
  \rho(t):=\sum_{e\in\E}\mis(\{x_e\in\I_e:u_e(x_e)>t\}),\quad t\geq0,
 \]
 is the \cor{distribution function} of $u$ (for more details, see \cite{AST-1, Friedlander}). Then, clearly,
 \begin{equation}
  \label{eq:riarrangp}
  \intd{\G}{\mo{u(x)}^r}{x}=\intd{\Rp}{\mo{\ru(x)}^r}{x}\quad\mbox{and}\quad\sup_{\G}u=\sup_{\Rp}\ru
 \end{equation}
 for every $r\geq1$. Moreover $\ru\in\hrpos$ and, by Proposition 3.1 in \cite{AST-1},
 \begin{equation}
  \label{eq:riarrangh}
  \intd{\Rp}{\mo{(\ru)'(x)}^2}{x}\leq\intd{\G}{\mo{u'(x)}^2}{x}.
 \end{equation}
 Then, from the Gagliardo-Nirenberg inequality on $\Rp$ (cfr. \cite{Agueh,Dolbeault})
 \begin{eqnarray*}
  \nlpg{u}^p & = & \nllp{\ru}^p\leq \Ce\nlldue{\ru}^{\f{p}{2}+1}\nlldue{(\ru)'}^{\f{p}{2}-1}\\
             & \leq & \Ce\nldueg{u}^{\f{p}{2}+1}\nldueg{u'}^{\f{p}{2}-1}.
 \end{eqnarray*}
 In the very same way it is possible to prove (\ref{eq:gni}).
\end{proof}

\begin{remark}
 \label{remark:gn_straight}
 It is well known that the \cor{best constant} of (\ref{eq:gni}) when $\G=\R$ is equal to 1. Indeed, assume that $u\in\hr$ satisfies $\mo{u(0)}=\max_{\R}\mo{u(x)}$ and let $\overline{u}$ be the \cor{even part} of $u$. Then
 \[
  u^2(0)=\overline{u}^2(0)=2\intl{-\iii}{0}{\overline{u}(x)\overline{u}'(x)}{x}\leq\nldue{\overline{u}}\nldue{\overline{u}'},
 \]
 so that
 \[
  \nli{u}\leq\nldue{u}^{\f{1}{2}}\nldue{u'}^{\f{1}{2}}.
 \]
 Since equality holds for $u=e^{-\mo{x}}$, the claim follows. With a slight further effort, one can also check that when $\G=\Rp$ the best constant is $\sqrt{2}$.
 
 According to these remarks, we see that for a generic graph $\G$ the best constant in (\ref{eq:gni}) satisfyies $\gn\leq\sqrt{2}$. In addition, if the graph contains at least two half-lines, then one can repeat the proof of Proposition \ref{proposition:gn} using the \cor{symmetric decreasing} rearrangement on $\R$ in place of the \cor{decreasing} one on $\Rp$ (see again \cite{AST-1, Friedlander}), thus obtaining that (in this case) $\gn\leq1$.
\end{remark}

Next we deal with the minimization problem (\ref{eq:constrainedminimi}) on a straight line with an extra Dirichlet condition.

\begin{proposition}
 \label{proposition:mindirichlet}
 Let $m>0,$ $a>0$. Then the function $\displaystyle u(x)=ae^{-\tf{a^2\mo{x}}{m}}$ satisfies
 \begin{equation}
  \label{eq:mindirichlet}
  \nldue{u'}^2=\min_{v\in\hdir}\nldue{v'}^2
 \end{equation}
 where $\hdir:=\{v\in\hrm:v(0)=a\}$.
\end{proposition}

\begin{proof}
 Consider inequality (\ref{eq:gni}). From Remark \ref{remark:gn_straight} we know that, when $\G=\R$, the best constant is $\gn=1$. Hence, for a non identically zero function $v$ in $\hr$, we find that
 \[
  \nldue{v'}^2\geq\f{\nli{v}^4}{\nldue{v}^2}
 \]
 and, assuming also that $v\in\hdir$,
 \begin{equation}
  \label{eq:infdirbound}
  \nldue{v'}^2\geq\f{a^4}{m}.
 \end{equation}
 Since $u\in\hdir$ and
 \[
  \nldue{u'}^2=\f{a^4}{m},
 \]
 we see that the infimum is attained by $u$.
\end{proof}

The following is an analogous result over the half-line.

\begin{corollary}
 \label{corollary:dirichlethalf}
  Let $m>0,$ $a>0$. Then the function $u(x):=ae^{-\tf{a^2x}{2m}}$, satisfies
  \begin{equation}
   \label{eq:dirichlethalf}
   \nlldue{u'}^2=\min_{v\in\hhdir}\nlldue{v'}^2.
  \end{equation}
\end{corollary}

Finally, we recall a standard result about the optimality conditions satisfied by any solution of (\ref{eq:constrainedminimi}) (for a proof, see \cite{AST-1}).

\begin{proposition}
 \label{proposition:eulageq}
 Let $\G$ be as in (\ref{eq:graphdef}) and let $\mu,p$ satisfy (\ref{eq:paramassum}). Suppose, in addition, that $u$ is a solution of (\ref{eq:constrainedminimi}). Then
 \begin{itemize}
  \item[(i)] there exists a real constant $\lambda$ such that, for every edge $e\in\E$
  \begin{equation}
   \label{eq:eulageq}
   u_e''+\kappa_eu_e\mo{u_e}^{p-2}=\lambda u_e\qquad\mbox{(Euler-Lagrange equation)}
  \end{equation}
  where $\kappa_e=1$ if $e\in\Ek$ and $\kappa_e=0$ if $e\in\E\backslash\Ek$;
  \item[(ii)] for every vertex $\vg\in\Vk$
  \begin{equation}
   \label{eq:kirchoff}
   \sum_{e\succ\vg}\der{u_e}{x_e}(\vg)=0\qquad\mbox{(Kirchhoff conditions)}
  \end{equation}
  where ``$e\succ\vg$'' means that the edge $e$ is incident at $\vg$;
  \item[(iii)] up to a change of sign, $u>0$ on $\G$.
 \end{itemize}
\end{proposition}

\begin{remark}
 \label{remark:kirchoff}
 The symbol $\der{u_e}{x_e}(\vg)$ is a shorthand notation for $u_e'(0)$ or $-u_e'(l_e)$, according to the fact that $x_e$ is equal to $0$ or $l_e$ at $\vg$.
\end{remark}

\section{Proof of the main results}\label{sec:proofs}

Throughout this section, it is convenient to identify each $u\in\hg$ with a $(\nn+1)$-ple of functions $\pp,\ff{1},\dots,\ff{\nn}$ such that
\begin{equation}
 \label{eq:reg_cond}
 \pp\in\hk,\quad\ff{i}\in\hrv{i}\quad\mbox{for}\quad i=1,\dots,\nn,
\end{equation}
and
\begin{equation}
 \label{eq:cont_cond}
 \pp(\w)=\ff{i}(\w)\quad\mbox{for}\quad i=1,\dots,\nn,
\end{equation}
where $\w$ denotes the starting vertex of $\Rv{i}$. Then (\ref{eq:constrainedminimi}) is equivalent to the minimization problem
\[
 \min\left(\frac{1}{2}\nlduegc{\pp'}^2-\frac{1}{p}\nlpgc{\pp}^p+\frac{1}{2}\sum_{i=1}^{\nn}\nldues{\ff{i}'}^2\right),
\]
with $(\pp,\ff{1},\dots,\ff{\nn})$ varying in the set of the $(\nn+1)$-ples satisfying (\ref{eq:reg_cond})-(\ref{eq:cont_cond}) and subject to the mass constraint
\begin{equation}
 \label{eq:mass_new}
 \nlduegc{\pp}^2+\sum_{i=1}^{\nn}\nldues{\ff{i}}^2=\mu.
\end{equation}

\medskip
\begin{proof}[Proof of Theorem~\ref{theorem:infsotto}]
 We prove the two statements separately.
 
 \medskip
 \noindent\cor{Part (i): proof of (\ref{eq:infzero}).} Consider a function $\xi\in C_0^{\iii}(\R^+)$ with $\nlldue{\xi}^2=\mu$. Then, for every $\lambda>0$ we define a function
 \[
  u_{\lambda}(x)=
  \left\{
  \begin{array}{ll}
   \sqrt{\lambda}\xi(\lambda x) & \mbox{if }x\in\Rv{1}\\
   0                            & \mbox{elsewhere.}  
  \end{array}
  \right.
 \]
 As a consequence, $u_{\lambda}\in\hm$ and $E(u_{\lambda},\G)=\f{\lambda^2}{2}\nlldue{\xi'}^2$. Hence
 \[
  \inf_{u\in\hm}E(u,\G)\leq\lim_{\lambda\rightarrow0}E(u_{\lambda},\G)=0
 \]
 and (\ref{eq:infzero}) is proved.
 
 \medskip
 \noindent \cor{Part (ii): (\ref{eq:infsotto}) entails existence.} Note that from (\ref{eq:gnp}) 
 \begin{equation}
  \label{eq:gnapplication}
  E(u,\G)\geq\f{1}{2}\nldueg{u'}^2-\f{\Ce\mu^{\f{p+2}{4}}}{p}\nldueg{u'}^{\tf{p}{2}-1}\quad\forall u\in\hm.
 \end{equation}
 Then, since $p/2-1<2$ by (\ref{eq:paramassum}), there results
 \begin{equation}
  \label{eq:funccoerc}
  \nhg{u}^2\leq C+CE(u,\G)
 \end{equation}
 (here $C$ depends on $\mu$). Consider now a minimizing sequence $u_k$ for problem (\ref{eq:constrainedminimi}). By (\ref{eq:funccoerc}) it is bounded in $\hg$ and therefore (up to subsequences)
 \[
  u_k\rightharpoonup u\quad\mbox{in}\quad\hg.
 \]
 Recalling that $u_k=(\pp_k,{\ff{1}}_k,\dots,{\ff{\nn}}_k)$ with $\pp_k,{\ff{1}}_k,\dots,{\ff{\nn}}_k$ satisfying (\ref{eq:reg_cond})-(\ref{eq:mass_new}), we see that there exist $\pp\in\hk$ and $\ff{i}\in\hrv{i}$, satisfying (\ref{eq:cont_cond}), such that
 \[
  \pp_k\rightharpoonup\pp\quad\mbox{in }\hk,\qquad{\ff{i}}_k\rightharpoonup\ff{i}\quad\mbox{in }\hrv{i}\quad\mbox{for } i=1,\dots,\nn.
 \]
 Moreover, since $\K$ is compact, we also have
 \[
  \pp_k\rightarrow \pp\quad\mbox{in}\quad\lpgc.
 \]
 Setting $u=(\pp,\ff{1},\dots,\ff{\nn})$, by weak lower semicontinuity,
 \begin{equation}
  \label{eq:wlsc}
  E(u,\G)\leq\liminf_k E(u_k,\G)=\inf_{v\in\hm}E(v,\G)<0
 \end{equation}
 and
 \[
  \nldueg{u}^2\leq\mu.
 \]
 If we prove, in addition, that the latter inequality is in fact an equality, then the function $u$ is a minimizer. Note that, by (\ref{eq:wlsc}), $u\not\equiv0$. Suppose now that $0<\nldueg{u}^2<\mu$. Hence there exists $\sigma>1$ such that $\nldueg{\sigma u}^2=\mu$. Consequently
 \begin{eqnarray*}
  E(\sigma u,\G) & = & \f{\sigma^2}{2}\intd{\G}{\mo{u'(x)}^2}{x}-\f{\sigma^p}{p}\intd{\K}{\mo{u(x)}^p}{x}\\
                 & < & \sigma^2E(u,\G)<E(u,\G)\leq\inf_{v\in\hm}E(v,\G),
 \end{eqnarray*}
 contradicting (\ref{eq:wlsc}). Then $\nldueg{u}^2=\mu$ and this proves (\ref{eq:infsotto}). Finally, $u>0$ immediately follows from Proposition (\ref{proposition:eulageq}).
\end{proof}

\medskip
\begin{proof}[Proof of Theorem~\ref{theorem:pdue}]
 Let $\lung=\mis(\K)$ and consider the function $u=(\pp,\ff{1},\dots,\ff{\nn})$ defined by
 \begin{equation}
  \label{eq:competitor}
  \pp\equiv a,\qquad\ff{i}(x)=ae^{-\tf{a^2x}{2m}}\quad\mbox{for }i=1,\dots,\nn,
 \end{equation}
 with
 \begin{equation}
  \label{eq:vinc_comp}
  a\in\left(0,\sqrt{\mu/\lung}\right)\quad\mbox{and}\quad m=\f{\mu-a^2\lung}{\nn}.
 \end{equation}
 Then (\ref{eq:reg_cond})-(\ref{eq:mass_new}) are satisfied and the energy functional reads
 \begin{equation}
  \label{eq:energy_constant}
  E(u,\G)=\frac{a^4\nn^2}{8(\mu-a^2\lung)}-\f{a^p\lung}{p}.
 \end{equation}
 Now, since $p\in(2,4)$, when $a$ is sufficiently small we see that $E(u,\G)<0$ and hence existence of minimizers for (\ref{eq:constrainedminimi}) follows from Theorem \ref{theorem:infsotto}.
\end{proof}

\medskip
\begin{proof}[Proof of Theorem~\ref{theorem:pquattro}]
 We break the proof in two parts.
 
 \medskip
 \noindent \cor{Part (i): proof of (\ref{eq:esistenza}).} Let $\lung=\mis(\K)$ and let $u$ be the function defined in (\ref{eq:competitor})-(\ref{eq:vinc_comp}), so that the energy functional reads as in (\ref{eq:energy_constant}). When $p=4$ one can easily check that, if
 \[
  \lung>\f{\nn^2}{2\mu},
 \]
 then there exists $a_0\in\left(0,\sqrt{\mu/\lung}\right)$ such that $E(u,\G)<0$. On the other hand, when $p\in(4,6)$ we claim that, if
 \[
  \lung>\Cp\f{\nn^{\f{4}{6-p}}}{\mu^{\f{p-2}{6-p}}}
 \]
 with
 \begin{equation}
  \label{eq:const_ex}
  \Cp=\left[\left(\f{p(p-4)}{16}\right)^{\f{2}{p-2}}+\f{p}{8}\left(\f{p(p-4)}{16}\right)^{\f{4-p}{p-2}}\right]^{\f{p-2}{6-p}},
 \end{equation}
 then there exists again $a_0\in\left(0,\sqrt{\mu/\lung}\right)$ such that $E(u,\G)<0$. Indeed, $E(u,\G)<0$ is equivalent to $g(a)<\mu$ where
 \[
  g(a)=a^2\lung+\f{\nn^2p}{8\lung}a^{4-p}.
 \]
 Easy computations show that $g$ has a unique critical point given by
 \[
  \overline{a}=\left(\f{\nn^2p(p-4)}{16\lung^2}\right)^{\f{1}{p-2}}.
 \]
 As $g$ is strictly convex, this is a global minimizer for $g$ in $(0,\sqrt{\mu/\lung})$ and, in addition,
 \[
  g(\overline{a})=\lung^{-\f{6-p}{p-2}}\left[\left(\f{\nn^2p(p-4)}{16}\right)^{\f{2}{p-2}}+\f{\nn^2p}{8}\left(\f{\nn^2p(p-4)}{16}\right)^{\f{4-p}{p-2}}\right].
 \]
 Now, if
 \[
  \overline{a}<\sqrt{\mu/\lung}\qquad\mbox{and}\qquad g(\overline{a})<\mu,
 \]
 then the claim is proved setting $a_0=\overline{a}$. The former inequality explicitly reads
 \begin{equation}
  \label{eq:a_min}
  \mu^{\f{p-2}{6-p}}\lung>\left(\f{\nn^2p(p-4)}{16}\right)^{\f{2}{6-p}},
 \end{equation}
 whereas the latter reads
 \begin{align}
  \label{eq:val_a_min}
  \mu^{\f{p-2}{6-p}}\lung & >\,\left[\left(\f{\nn^2p(p-4)}{16}\right)^{\f{2}{p-2}}+\f{\nn^2p}{8}\left(\f{\nn^2p(p-4)}{16}\right)^{\f{4-p}{p-2}}\right]^{\f{p-2}{6-p}}\nonumber\\
  & =\,N^{\f{4}{6-p}}\Cp.
 \end{align}
 Now, observing that (\ref{eq:val_a_min}) entails (\ref{eq:a_min}), letting
 \begin{equation}
  \label{eq:const_exist}
  \thes=\left\{
  \begin{array}{ll}
   \displaystyle\tf{1}{2}\mu^{-1}\nn^2                & \mbox{if }p=4\\[.3cm]
   \displaystyle\Cp\mu^{\f{2-p}{6-p}}\nn^{\f{4}{6-p}} & \mbox{if }p\in(4,6)
  \end{array}
  \right.
 \end{equation}
 we see that (\ref{eq:esistenza}) follows from Theorem \ref{theorem:infsotto}.
 
 \medskip
 \noindent \cor{Part (ii): proof of (\ref{eq:nonesistenza}).} Suppose that there exists a function $u\in\hm$ such that
 \begin{equation}
  \label{eq:assurdipotesi}
  E(u,\G)\leq0.
 \end{equation}
 We claim that
 \begin{equation}
  \label{eq:inductiveineq}
  \nldueg{u'}^2\leq\f{1}{\gn^4\mu}\nlig{u}^{4\left(\tf{p}{4}\right)^{n+1}}(\gn^4\mu\lung)^{\sum_{i=0}^n\left(\tf{p}{4}\right)^i}\quad\forall n\geq0
 \end{equation}
 (once again $\lung=\mis(\K)$). To check this by induction, we first note that from (\ref{eq:assurdipotesi}) we have
 \begin{equation}
  \label{eq:inducbase}
  \nldueg{u'}^2\leq\f{2}{p}\nlpgc{u}^p\leq\lung\nlig{u}^p,
 \end{equation}
 since $2/p<1$, which corresponds to (\ref{eq:inductiveineq}) when $n=0$. Moreover, for fixed $n>0$, assume that
 \begin{equation}
  \label{eq:inducass}
  \nldueg{u'}^2\leq\f{1}{\gn^4\mu}\nlig{u}^{4\left(\tf{p}{4}\right)^n}(\gn^4\mu\lung)^{\sum_{i=0}^{n-1}\left(\tf{p}{4}\right)^i}.
 \end{equation}
 Combining with (\ref{eq:gni}), we see that
 \begin{equation}
  \label{eq:inducstep}
  \nlig{u}^p\leq\nlig{u}^{4\left(\tf{p}{4}\right)^{n+1}}(\gn^4\mu\lung)^{\sum_{i=1}^n\left(\tf{p}{4}\right)^i}
 \end{equation}
 and plugging into (\ref{eq:inducbase}) we recover (\ref{eq:inductiveineq}). Furthermore, combining (\ref{eq:assurdipotesi}) with (\ref{eq:gnapplication}) and (\ref{eq:gni}), we find that
 \[
  \nlig{u}\leq\gn\Ce^{\f{1}{6-p}}\mu^{\f{2}{6-p}}
 \]
 and thus, putting into inequality (\ref{eq:inductiveineq}), we obtain
 \[
  \begin{array}{c}
   \nldueg{u'}^2\leq\Ce^2\mu^3(\gn^4\mu\lung)^{n+1},\qquad\mbox{if }p=4,\\[.5cm]
   \nldueg{u'}^2\leq\Ce^{\f{4}{6-p}}\mu^{\f{p+2}{6-p}}\left(\gn^{\f{4p}{p-4}}\Ce^{\f{4}{6-p}}\mu^{\f{4(p-2)}{(p-4)(6-p)}}\lung^{\f{4}{p-4}}\right)^{\left(\f{p}{4}\right)^{n+1}-1},\quad\mbox{if }p\in(4,6),
  \end{array}
 \]
 for every $n\geq0$. Now, when the terms in brackets in the two inequalities are strictly smaller than 1, letting $n\rightarrow\iii$ we obtain that $\nldueg{u'}^2=0$, which is a contradiction since $\nldueg{u}^2=\mu>0$. Consequently, if we choose
 \[
  \thnon=\Ce^{\f{4-p}{6-p}}\mu^{\f{2-p}{6-p}}\gn^{-p},
 \]
 then we see that $E(u,\G)>0$ whenever $\lung<\thnon$ and hence (\ref{eq:nonesistenza}) follows from Theorem \ref{theorem:infsotto}.
\end{proof}

\medskip
\begin{proof}[Proof of Corollary~\ref{corollary:nonesistenza}]
 Consider a function $u$ in $\hm$. According to the partition $(\G_i)_{i=1}^{\pa}$, the energy functional reads
 \begin{equation}
  \label{eq:energy_dec}
  E(u,\G)=\sum_{i=1}^{\pa}E(u_{|_{\G_i}},\G_i).
 \end{equation}
 Now if $u_{|_{\G_i}}\equiv0$ for some $i$, then $E(u_{|_{\G_i}},\G_i)=0$. On the other hand, if $u_{|_{\G_i}}\not\equiv0$, since
 \[
  \|{u_{|_{\G_i}}}\|_{L^2(\G_i)}^2\leq\mu,
 \]
 then there exists $\sigma\geq1$ satisfying
 \[
  \intd{\G_i}{\sigma^2\mo{u(x)}^2}{x}=\mu,
 \]
 so that
 \[
  E(\sigma u_{|_{\G_i}},\G_i)\leq\sigma^2E(u_{|_{\G_i}},\G_i).
 \]
 Therefore, since $\mis(\K_i)<\thnon$, (\ref{eq:nonesistenza}) entails that $E(\sigma\left.u\right|_{\G_i},\G_i)>0$, whence $E(\left.u\right|_{\G_i},\G_i)>0$. Now, as $u\in\hm$, then $u_{|_{\G_i}}\not\equiv0$ for at least one $i$. As a consequence, $E(u,\G)>0$, which concludes the proof.
 
 \noindent Note that we omit the case of a function $u\in\hm$ with support contained in a subgraph of the partition having trivial compact part, since this immediately entails that $E(u,\G)>0$.
\end{proof}

\bigskip
\noindent\gr{Acknowledgements}

\medskip
\noindent The author is grateful to Enrico Serra for many helpful suggestions during the preparation of this work and to Paolo Tilli and Riccardo Adami for several enlightening discussions.

\end{document}